\documentclass[11pt]{amsart}
\usepackage{color, soul}
\usepackage{amsmath,amsfonts,amssymb}

\newtheorem{theorem}{Theorem}[section]
\newtheorem*{theorem*}{Theorem}
\newtheorem{proposition}[theorem]{Proposition}
\newtheorem{corollary}[theorem]{Corollary}
\newtheorem{lemma}[theorem]{Lemma}
\theoremstyle{definition}
\newtheorem{definition}[theorem]{Definition}

\newtheorem{remark}[theorem]{Remark}

\begin{document}
\baselineskip=.65cm

\title[s-numbers  sequences for homogeneous polynomials]{s-numbers sequences for homogeneous polynomials}

\author[E. \c Cal\i \c skan]{Erhan \c Cal\i \c skan}
\address[E. \c Cal\i \c skan]{Y\i ld\i z Technical University, Faculty of Sciences and Arts,
Department of Mathematics,
Davutpa\c sa, 34210 Esenler, \.Istanbul, Turkey} \email{ercalis@yahoo.com.tr}

\author[P. Rueda]{Pilar Rueda}
\address[P. Rueda]{Departamento de An\'{a}lisis Matem\'{a}tico,
Universidad de Valencia, Doctor Moliner 50, 46100 Burjasot (Valencia), Spain}
\email{pilar.rueda@uv.es}

\thanks{The second author was supported by MICINN  MTM2011-22417.}



\date{\today}

\begin{abstract}We extend the well known theory of $s$-numbers of linear operators to homogeneous polynomials defined between Banach spaces.
  Approximation, Kolmogorov and Gelfand numbers of polynomials are introduced and some well-known results of the linear and multilinear settings are obtained for homogeneous polynomials.
\end{abstract}

\maketitle

\section{Introduction}

A. Pietsch \cite{AP}  introduced $s$-numbers as  a tool for the study of linear operators  between Banach spaces. The success of the linear operator theory gave rise to consider a  multilinear and polynomial analogue that was proposed firstly by A. Pietsch and followed by many researchers in the last decades (see \cite{BC} and the references therein). The theory of $s$-numbers of multilinear operators among Banach spaces has been recently developed by D. L. Fernandez, M. Masty\-lo and E. B. da Silva \cite{DLFMMEBS}. While the properties of $s$-numbers of linear operators are well-known, the analogous theory  of homogeneous polynomials has not
been checked as far as it should have been. We mention only the particular case by
A. Brauns, H. Junek and E. Plewnia  \cite{BrJuPl} and the unpublished  \cite{ABHJ}.

The aim of this paper is to elaborate the corresponding theory of $s$-numbers in the context of homogeneous polynomials. It is worth mentioning that in many situations  dealing with polynomials instead of multilinear mappings has proved to be a subtle subject that has needed different approaches and that has yielded to different results. For instance, when trying to generalize absolutely summing operators to a non linear context, different approaches have been required, and whereas factorization theorems have been stated in the multilinear case, the search for a factorization scheme for dominated polynomials has turned out to be difficult and only partial results have been obtained (see \cite{BoPeRuJFA, BoPeRuLAMA2012, BoPeRuLAMA2014}. The main purpose of the present paper is to undertake a study of the basics on $s$-number sequences for polynomials, that include: approximation numbers $\tilde a_n$ and their relation with the adjoint and the biadjoint of a homogeneous polynomial introduced by  Aron and Schottenloher; Kuratowski and Hausdorff measures  of non-compactness of homogeneous polynomials; Kolmogorov numbers $\tilde d_n$ and their relation with the approximation numbers; Gelfand numbers $\tilde c_n$, the equivalence between $P$ being compact and $\lim_{n\to \infty}\tilde c_n(P)=0$ and their relation with Kolmogorov numbers. Some of the proofs we present are inspired by the linear/multilinear ones, whereas other use techniques from polynomial theory. Our aim is to present the theory from the point of view of polynomials and relate it to the linear or multilinear case with linearization techniques that will provide shorter proofs than coming from the classical theory. Furthermore, the results obtained for homogeneous polynomials can be considered extensions of the linear ones.

Section 2 is devoted to fix  notation and state some basic definitions and preliminary results. In a quite natural way, we introduce the notion of $m-s$-number sequence for $m$-homogeneous polynomials and relate it with the classical $s$-number sequences of linear operators.
 Section 3 contains the essentials of the $n$-th approximation number $\tilde a_n(P)$ of an homogeneous polynomial $P$ and its coincidence with the $n$-th approximation number of the adjoint $P^*$  of  $P$ in the sense of \cite{RMAMS}.
Compactness of homogeneous polynomials is treated in Section 4 by means of Kuratowski and Hausdorff measures $\gamma$ and $\tilde \gamma$ respectively.  We prove that $\gamma(P)\leq \tilde \gamma(P^*)$ and $\gamma(P^*)\leq \tilde \gamma(P)$ among other inequalities. As an application, we recover the well-known result that a homogeneous polynomial is compact if and only if its adjoint is compact. As an attempt to quantify the non compactness character of a polynomial, we study the polynomial notion of Kolmogorov numbers $\tilde d_n$ and the polynomial $m$-lifting property. In particular we prove that $\tilde d_n(P)=\tilde d_n(P_L)$, where $P_L$ is the linearization of an $m$-homogeneous polynomial $P$ defined on a Banach space $X$. Moreover, $\tilde d_n(P)=\tilde a_n(PQ)$, where $Q$ is the canonical metric surjection from $\ell_1(\overline B_X)$ onto $X$ defined by
${\hspace{0cm}}Q(\{\lambda_x\})=\sum_{x\in\overline{B}_X}\lambda_xx$, $\{\lambda_x\}\in l_1(\overline{B}_X)$.
 Finally, we deal with Gelfand's numbers $\tilde c_n$ adapted to the polynomial context. We obtain characterizations of compactness of homogeneous polynomials, this time, in terms of Gelfand numbers, and we prove that $\tilde c_n(P^*)\leq \tilde d_n(P)$, $\tilde c_n(P)=\tilde d_n(P^*)$ and $\tilde c_n(P)\leq 2\sqrt{n}\tilde c_n(P^*)$.

\section{Notation and preliminaries}

The symbol $\mathbb{K}$ represents the field of all real numbers or complex numbers,
${\mathbb{N}}$ represents the set of all positive integers.

The letters $X$, $Y$ and $Z$ will always represent (real or complex) Banach spaces.
The symbol
$B_X$ represents the open unit ball of $X$ and $\overline{B}_X$
 the closed unit ball. We denote by $X^*$ the dual Banach space of $X$, and by $\kappa_X$ the canonical embedding of $X$ into the bidual $X^{**}$ of $X$.

Given a subset $C\subset X$, let $\overline \Gamma (C)$ denote the closed balanced convex hull of $C$.

Let $1\leq p\leq\infty$, with the conjugate index $p'$ given by $\displaystyle \frac{1}{p}+\frac{1}{p'}=1$ (where $p'=1$ if $p=\infty$), let $\ell_p(X)$ ($1\leq p<\infty$) (resp., $\ell_{\infty}(X)$) denote the set of all sequences $(x_n)_{n=1}^\infty$ in $X$ such that $\displaystyle \sum_{n=1}^{\infty}\| x_n\|^{p}<\infty$ (resp., $(x_n)_n$ is bounded), and let $c_0(X)$ denote  the set of all sequences $(x_n)_{n=1}^\infty$ in $X$ such that $x_n\longrightarrow0$ in $X$.

Given a continuous $m$-linear mapping $A:X \times \cdots\times X \to Y$, the map
$P:X\longrightarrow Y$, defined by
$P(x)=A(\underbrace{x,\ldots,x}_{m\ times})$ for every $x\in X$, is
said to be a continuous $m$-homogeneous polynomial.
${\mathcal{P}}(^mX;Y)$ will denote the vector space of all
continuous $m$-homogeneous polynomials from $X$ into $Y$, which is a
Banach space with norm $\| P\|=\sup\{\|
P(x)\| : \
\| x\|\leq 1\}$. When $Y=\mathbb{K}$ we will write ${\mathcal{P}}(^mX)$
instead of ${\mathcal{P}}(^mX;\mathbb{K})$ and when $m=1$, ${\mathcal L}(X;Y):={\mathcal P}(^1X;Y)$ is the space of all continuous linear operators from $X$ to $Y$.  Let ${\mathcal P}^m:=\bigcup_{X,Y} {\mathcal P}(^mX;Y)$, that is, ${\mathcal P}^m$ is the class of all $m$-homogeneous polynomials defined between Banach spaces. Denote by ${\mathcal P}:=\bigcup_m{\mathcal P}^m$ the class of all continuous homogeneous  polynomials defined between Banach spaces.

Let $P\in\mathcal{P}(^mX;Y)$. We define the rank of $P$ as the dimension of the linear span of $P(X)$ in $Y$: $rank(P)=dim([P(X)])$.


In a natural way, we introduce the notion of an $m-s$-number sequence for $m$-homogeneous continuous polynomials.
Let $m\in {\mathbb{N}}$ and for each $n\in \mathbb N$ let $s_n: {\mathcal{P}}^m\longrightarrow [0,\infty)$ be a mapping. The sequence $s=(s_n)$  is called an {\it $m-s$-number sequence} if the following conditions are satisfied for any $n,k\in \mathbb N$:

$(S1)$ Monotonicity:  For every $P\in{\mathcal{P}}(^mX;Y)$,

${\hspace{3.5cm}}\| P\|=s_1(P)\geq s_2(P)\geq\ldots\geq0$.

$(S2)$ Additivity:  For every $P,Q\in{\mathcal{P}}(^mX;Y)$,

${\hspace{3.5cm}}s_{k+n-1}(P+Q)\leq s_k(P)+s_n(Q)$.

$(S3)$ Ideal-property:  For every $P\in{\mathcal{P}}(^mX;Y)$, $S\in{{\mathcal L}}(Y;Z)$, $T\in{{\mathcal L}}(W;X)$

${\hspace{3.5cm}}s_n(SPT)\leq \| S\| s_n(P)\| T\|^m$.

$(S4)$ Rank-property:  Let $P\in{\mathcal{P}}(^mX;Y)$.

${\hspace{3.5cm}} rank(P)<n \Longrightarrow s_n(P)=0$.

\noindent
Furthermore, if $m=1$ the following condition has to be added:

$(S5)$ Norming-property: $s_n({\rm Id}:\ell_2^n\to \ell_2^n)=1$, $n\in \mathbb N$, where Id is the identity mapping on the $n$-dimensional Hilbert space $\ell_2^n$.

If $(s_n)$ is an $m-s$-number sequence for each $m\in\mathbb{N}$, then $(s_n)$ is called an {\it $s$-number sequence.} Note that this notion coincides with the usual notion of $s$-number sequence for linear operators whenever $m=1$.

Given $P\in {\mathcal P}(^mX;Y)$, let $P_{L}$ denote the linearization of $P$; that is the unique continuous linear operator $P_L:\hat\otimes_{m,s}^{\pi_s}X\to Y$ such that $P(x)=P_L(\otimes_m x)$. The correspondence $P\leftrightarrow P_L$ determines an isometric isomorphism ---denoted by $I_{^mX,Y}$---  between ${\mathcal P}(^mX;Y)$ and the space  $L(\otimes_{m,s}^{\pi_s}X;Y)$ of all continuous linear operators from $\otimes_{m,s}^{\pi_s}X$ to $Y$.  Let $I_m:{\mathcal P}^m\longrightarrow {\mathcal L}$ and $I:{\mathcal P}\to {\mathcal L}$  be the correspondences whose restrictions to each component ${\mathcal P}(^mX;Y)$ is equal to $I_{^mX,Y}$.

If $T\in {\mathcal L}(X;Y)$, let $\otimes_{m,s}T:\otimes_{m,s}^{\pi_s}X\to \otimes_{m,s}^{\pi_s}Y$ be the continuous linear map given by $\otimes_{m,s}T(\otimes_m x)=\otimes_mT(x)$.

Our first interest is to relate the linear and the polynomial notions of $s$-number sequences.
\begin{proposition}\label{I}
If the mapping $s=(s_n):{\mathcal  L}\longrightarrow [0,\infty)^{{\mathbb{N}}}$ is an $s$-number sequence (in the linear sense) then, $s\circ I_m: {\mathcal{P}}^m\longrightarrow [0,\infty)^{{\mathbb{N}}}$ is an  $m-s$-number sequence.
\end{proposition}
\begin{proof} We will pay attention just to the ideal property.
This property follows from the fact that $(SPT)_L=SP_L\otimes_{m,s}T$ and $\|\otimes_{m,s}T\|=\|T\|^m$, for all
 $P\in\mathcal{P}(^mX;Y)$, $T\in{\mathcal L}(W;X)$ and $S\in{\mathcal L}(Y;Z)$.
\end{proof}

Injectivity, surjectivity and multiplicativity have been proved useful tools for $s$-number sequences. Let us extend these properties to  the polynomial context.

$(J)$ An $m-s$-number sequence $s=(s_n)$ is called {\it injective} if  given any metric
injection $j\in{\mathcal L}(Y;Z)$, i.e., $\| j(y)\|=\| y\|$ for all $y\in Y$, $s_n(P)=s_n(jP)$ for all $P\in{\mathcal{P}}(^mX;Y)$ and for all Banach spaces $X$.

$(S)$ An $m-s$-number sequence $s=(s_n)$ is called {\it surjective} if given any metric surjection $q\in{\mathcal L}(Z;X)$, i.e., $q(B_Z)=B_X$, $s_n(P)=s_n(Pq)$ for all $P\in{\mathcal{P}}(^mX;Y)$.

$(M)$ An $s$-number sequence $s=(s_n)$ is called multiplicative if, for $u\in{\mathcal L}(Y;Z)$ and $P\in{\mathcal{P}}(^mX;Y)$,

${\hspace{3.5cm}}s_{k+n-1}(u\circ P)\leq s_k(u)s_n(P)$, $k,n\in\mathbb{N}$.

\begin{proposition}\label{injsur}
Let   $s=(s_n):{\mathcal L}\longrightarrow [0,\infty)^{\mathbb N}$ be an $s$-number sequence for linear operators.
\begin{enumerate}
\item  If $s$ is injective then $s\circ I_m$ is injective.
\item If $s$ is surjective then, $s\circ I_m$ is surjective.
\item If $s$ is multiplicative then $s\circ I_m$ is multiplicative.
\end{enumerate}
\end{proposition}

\begin{proof} (1)  Let  $j\in{\mathcal L}(Y;Z)$ be a metric injection and $P\in {\mathcal P}(^mX;Y)$. Then,

${\hspace{1cm}}s_n\circ I_m(P)=s_n(P_L)=s_n(j\circ P_L)=s_n((j\circ P)_L)=s_n\circ I_m(j\circ P).$

\noindent (2) Let $q\in{\mathcal L}(Z;X)$ be a metric surjection and $P\in{\mathcal{P}}(^mX;Y)$. If $\otimes q:\otimes_{m,s}^{\pi_s}Z\longrightarrow \otimes_{m,s}^{\pi_s}X$ denotes the linear map given by $\otimes q(\otimes x):=\otimes q(x)$, then

${\hspace{1cm}}\otimes q( B_{\otimes_{m,s}^{\pi_s}Z})=\otimes q(\Gamma (\otimes_{m,s}^{\pi_s}B_Z))= \Gamma(\otimes_{m,s}^{\pi_s}q(B_Z))=\Gamma(\otimes_{m,s}^{\pi_s}B_X)= B_{\otimes_{m,s}^{\pi_s}X}.$

\noindent Hence, for any $P\in {\mathcal P}(^mX;Y)$ and any metric surjection $q\in{\mathcal L}(Z;X)$ we have

${\hspace{2cm}}s_n\circ I_m(P\circ q)=s_n((P\circ q)_L)=s_n(P_L\circ \otimes q)=s_n( P_L)=s_n\circ I_m( P)$.

\noindent (3) Let $u\in{\mathcal L}(Y;Z)$ and $P\in{\mathcal{P}}(^mX;Y)$. Then,

${\hspace{1cm}}s_{k+n-1}\circ I_m(u\circ P)=s_{k+n-1}((u\circ P)_L)=s_{n+k-1}(u\circ P_L)\\
{\hspace{4.5cm}}\leq s_k(u) s_n(P_L)=s_k\circ I_m(u)s_n\circ I_m(P).$
\end{proof}

The theory of ideals of homogeneous polynomials between Banach spaces has been developed in the last decades by several authors, so the extension of the dual procedure to polynomial ideals is a natural step. In this paper we provide many results on homogeneous polynomials in connection with their adjoint concerning measure of non-compactness and $s$-numbers. First we need the definition of the adjoint of a continuous homogeneous polynomial.
\begin{definition}\rm (Aron--Schottenloher \cite{RMAMS}) Given a continuous $m$-homogeneous polynomial $P \in {\mathcal{P}}(^mX;Y)$ between the Banach spaces $X$ and $Y$, the adjoint of $P$ is the following continuous linear operator:

${\hspace{3.5cm}}P^* \colon Y^* \longrightarrow {\mathcal{P}}(^mX)~,~P^*(\varphi)(x) =\varphi(P(x)).$

\noindent It is clear that $\|P^*\| = \|P\|$.
\end{definition}

After this definition by R. Aron and M. Schottenloher, and after the works of R. Ryan \cite{RR2, RR1}, the adjoint of a polynomial became a standard tool in the study of spaces of homogeneous polynomials and in infinite dimensional holomorphy (see, e.g. \cite{SD, SDJM, JM2} and references therein).
We refer to \cite{SD} or
\cite{JM} for the properties of polynomials in infinite dimensional spaces, and to \cite{JLLT} for the
theory of Banach spaces.

\section{Approximation numbers of homogeneous polynomials}

Similar to the linear case we define the $n$-th approximation number $\tilde a_n(P)$ of any homogeneous polynomial $P\in\mathcal{P}(^mX;Y)$ by

${\hspace{2cm}}\tilde a_n(P):=inf\{\|P-Q\|: \ Q\in\mathcal{P}(^mX;Y), \ rank(Q)<n\}$.

\noindent If we denote $a_n(T):=\inf\{ \|T-L\|: \ L\in {\mathcal L}(X;Y), {\rm rank}(L)<n\}$, $T\in {\mathcal L}(X;Y)$, then  $\tilde a_n(P)=a_n(P_L)$.

If  $a=(a_n)$ is an $s$-number sequence on $\mathcal L$, Proposition \ref{I} gives that
 $\tilde a=a\circ I_m$ is an $m-s$-number sequence on ${\mathcal P}^m$. Therefore,  $\tilde a=a\circ I$  is an $s$-number sequence.


\begin{proposition}\label{p:090.11.0}  Let $(s_n):\mathcal{P}(^mX;Y)\longrightarrow [0,\infty)^{\mathbb{N}}$ be an $s$-number sequence. Then

(i) For all $P\in\mathcal{P}(^mX;Y)$ we have  $s_n(P)\leq \tilde a_n(P)$, $n\in\mathbb{N}$.

(ii) For all $S\in{\mathcal L}(Y;Z)$, $P\in\mathcal{P}(^mX;Y)$ and all $k,n\in\mathbb{N}$ we have $s_{k+n-1}(SP)\leq s_1(S)\tilde a_n(P)$ and $s_{k+n-1}(SP)\leq a_k(S)s_n(P)$.
\end{proposition}
\begin{proof}  (i) Let $P\in\mathcal{P}(^mX;Y)$. Then for any $R\in\mathcal{P}(^mX;Y)$ with $rank(R)<n$, we have
${\hspace{0cm}}s_n(P)\leq s_1(P-R)+s_n(R)=\| P-R\|+s_n(R)=\| P-R\|$.
Hence, by definition of $\tilde a_n(P)$ we have $s_n(P)\leq \tilde a_n(P)$.

\noindent (ii) Let $R\in\mathcal{P}(^mX;Y)$ with $rank(R)<n$. Since $rank(SR)<n$, it follows that

${\hspace{1.5cm}}s_{k+n-1}(SP)\leq s_k(S(P-R))+s_n(SR)=s_k(S(P-R))$

${\hspace{3.5cm}}\leq\| S\| s_k(P-R)\| I_X\|^m\leq\| S\| s_1(P-R)=\| S\|\| P-R\|$.

\noindent Hence by definition of $(\tilde a_n)$ we get $s_{k+n-1}(SP)\leq s_1(S)\tilde a_n(P)$.

The proof of the second inequality can be obtained in a similar way.
\end{proof}


\begin{remark}\label{l:60.11.90} If $P\in\mathcal{P}(^mX;Y)$ has finite rank then $P_L$ has finite rank. Hence,
${\rm rank}(P)={\rm rank}( P_L)={\rm rank}((P_L)^*)={\rm rank}(I_{^mX,\mathbb K}\circ P^*)={\rm rank}(P^*)$.
\end{remark}

It is worth mentioning that our use of polynomial techniques allows us to reduce to the linear case many proofs instead of adapting all calculations to the new setting.

\begin{proposition}\label{p:07.07.09} Let $m\geq2$ and let $X$ and $Y$ be Banach spaces. For every polynomial $P\in \mathcal{P}(^mX;Y)$ we have
$a_n(P^*)\leq \tilde a_n(P)$, $n\in\mathbb{N}$. Furthermore, if there exists a linear projection $\pi$ of norm $1$ from $Y^{**}$ onto $\kappa_Y(Y)$ then, for every $P\in\mathcal{P}(^mX;Y)$ we have that $a_n(P^*)=\tilde a_n(P)$, $n\in\mathbb{N}$.
\end{proposition}
\begin{proof} Since $I_{^mX,{\mathbb K}}$ is an isometric isomorphism, it follows from, e.g., \cite[p. 152, 11.7.3. Proposition]{AP2} that

${\hspace{2cm}}a_n(P^*)=a_n(I_{^mX,{\mathbb K}}\circ P^*)=a_n((P_L)^*)\leq a_n(P_L)=\tilde a_n(P).$

For the second assertion, we use the analogous property for linear operators \cite[Proposition 3.3]{DLFMMEBS} to get that

${\hspace{2cm}}\tilde a_n(P)=a_n(P_L)=a_n(P_L^*)=a_n(I_{^mX,{\mathbb K}}\circ P^*)=a_n(P^*).$
\end{proof}

\begin{remark}\label{remark} The technique we have used in this section makes use of the linearization of continuous homogeneous polynomials. A similar technique works for continuous $m$-linear mappings. For each integer $m\in\mathbb{N}$, let ${\mathcal L}(X_1,\ldots,X_m;Y)$ be the
Banach space of all continuous $m$-linear mappings
$A:X_1\times\ldots\times X_m\longmapsto Y$, endowed with the sup
norm $\| A\|=\sup\{\|
A(x_1,\ldots,x_m)\| : \
\| x_i\|\leq 1,\ i=1,\ldots,m\}$.
 If $T\in {\mathcal L}(X_1,\ldots,X_m;Y)$ there is a unique continuous linear operator $T_L\in{\mathcal L}(X_1\hat \otimes_{\pi}\cdots \hat \otimes_{\pi} X_m;Y)$ such that $T_L(x_1\otimes\ldots\otimes x_m)=T(x_1,\ldots,x_m)$, and the correspondence $T\leftrightarrow T_L$ determines an isometric isomorphism between ${\mathcal L}(X_1,\ldots,X_m;Y)$ and ${\mathcal L}(X_1\hat \otimes_{\pi}\cdots \hat \otimes_{\pi} X_m;Y)$. This could yield to alternative proofs in \cite{DLFMMEBS} based in the well-known linear case.
\end{remark}

\section{Compactness of homogeneous polynomials}

The results in this section shows that the natural extensions of Kuratowski and Hausdoff measures to polynomials keeps the harmony between linear and non linear theory.

Let $\mathcal{X}$ be a  metric space. The { \it Kuratowski measure} $\alpha(A)$ of non-compactness of a bounded set $A\subset \mathcal X$ is defined by

$\alpha(A)=inf\{\varepsilon>0: \text{A may be covered by finitely many sets of diameter}\leq\varepsilon\}$.

\noindent In case that we consider just finitely many balls of radius $\leq  \varepsilon$ to cover $A$, the infimum is called the {\it Hausdorff ball measure} $\beta(A)$ of non-compactness of $A$, that is

$\beta(A)=inf\{\varepsilon>0: \text{A may be covered by finitely many balls of radius}\leq\varepsilon\}$.

\noindent For every bounded set $A$ we have that $\beta(A)\leq\alpha(A)\leq2\beta(A)$.

Let $X$ and $Y$ be Banach spaces. Since continuous $m$-homogeneous polynomials are bounded on bounded sets, we can extend the Kuratowski, and the Hausdorff measure of non-compactness of linear operators to polynomials in a natural way:  for any $P\in{\mathcal{P}}(^mX;Y)$  the {\it Kuratowski and the Hausdorff measure}, respectively, of non-compactness of $P$ is defined by

${\hspace{2cm}}\gamma(P):=\alpha(P(\overline{B}_X)) \ \ \mbox{ and }\ \ \  \widetilde{\gamma}(P):=\beta(P(\overline{B}_X))$

Note  that $P$ is compact if and only if $\widetilde \gamma (P)=\gamma(P)=0$.

\begin{remark} (see \cite[Theorem 2.9]{DEEWDE}) Suppose $X$ and $Y$ are Banach spaces and let $T\in{\mathcal L}(X;Y)$. Then $\gamma(T)\leq\widetilde{\gamma}(T^*)$ and $\gamma(T^*)\leq\widetilde{\gamma}(T)$.
\end{remark}

\begin{lemma}\label{epsilon}
Let $X$ be a Banach space. Then $\beta(C)=\beta(\overline \Gamma(C))$ for any  $C\subset X$. In particular, $\beta(\overline \Gamma(P(B_X)))=\tilde \gamma(P)$.
\end{lemma}

\begin{proof}  Let $\epsilon,\delta>0$. It suffices to be shown that if $C$ can be covered by finitely many balls of radius $\epsilon$ then, $\overline \Gamma(C)$ can be covered by finitely many balls of radius $\delta+\epsilon$.

Assume that there are $x_1,\ldots,x_N\in X$ such that $C\subset \cup_{i=1}^N x_i+\epsilon B_X$. Take $i,j\in \{1,\ldots,N\}$. Since the set $C_{i,j}:=\{\alpha x_i+\beta x_j:\ |\alpha| +|\beta|= 1\}$ is compact, there are $z_{i,j}^1,\ldots,z_{i,j}^{L_{i,j}}\in X$ such that $C_{i,j}\subset \cup_{k=1}^{L_{i,j}}(z_{i,j}^k+\delta B_X)$.
 If $\tilde x_i\in x_i+\epsilon B_X$, $\tilde x_j\in x_j+\epsilon B_X$ and $k\in \{1,\ldots,L_{i,j}\}$ is such that $\alpha x_i+\beta x_j\in z_{i,j}^k+\delta B_X$ then,
\begin{eqnarray*}
\|\alpha\tilde x_i+\beta \tilde x_j-z_{i,j}^k\|&\leq &|\alpha| \|\tilde x_i-x_i\|+\|\alpha x_i+\beta x_j-z_{i,j}^k\|+|\beta|\|\tilde x_j-x_j\|\\
&\leq & |\alpha | \epsilon+\delta+|\beta|\epsilon=\delta+\epsilon.
\end{eqnarray*}
Hence,

$
\overline \Gamma(\cup_{i=1}^N x_i+\epsilon B_X)\subset \overline{\cup_{i,j=1}^N(\cup_{k=1}^{L_{i,j}}z_{i,j}^k+(\delta +\epsilon) B_X)} \subset \cup_{i,j=1}^N(\cup_{k=1}^{L_{i,j}}z_{i,j}^k+(\delta +\epsilon)\overline{ B_X}).
$
\end{proof}

Consider the
$m$-homogeneous polynomial $J:X\longrightarrow \mathcal{P}(^mX)^*$  given by
$J(x)(B):=B(x)$, $x\in X$, $B\in \mathcal{P}(^mX)$. Since
$P^{**}\circ J=\kappa_Y\circ P$, the following theorem can be proved also with similar techniques to the ones given in \cite[Theorem 2.1]{DLFMMEBS}. However, we will use polynomial techniques related to tensor products to show that the polynomial case admits shorter proofs.

\begin{theorem} \label{t:1.1.1} Let $m\geq2$ and let $X$ and $Y$ be Banach spaces. Then
\begin{enumerate}
\item  $\gamma(P)\leq\widetilde{\gamma}(P^*)$ and $\gamma(P^*)\leq\widetilde{\gamma}(P)$,
\item $\frac{1}{2}\gamma(P)\leq\gamma(P^*)\leq2\gamma(P)$ and $\frac{1}{2}\widetilde{\gamma}(P)\leq\widetilde{\gamma}(P^*)\leq2\widetilde{\gamma}(P)$.
\end{enumerate}
for every $P\in {\mathcal{P}}(^mX;Y)$.
\end{theorem}

\begin{proof} (1) Since

${\hspace{2cm}}P_L(\overline B_{\otimes_{m,s}^{\pi_s}X})=P_L(\overline \Gamma(\otimes_{m,s}B_X))=\overline \Gamma(P_L(\otimes_{m,s}B_X))=\overline\Gamma (P(B_X)),$

\noindent we obtain

${\hspace{2cm}}\gamma(P_L)=\alpha(P_L(B_{\otimes_{m,s}^{\pi_s}X}))\geq \alpha(P(B_X))=\gamma(P).$

\noindent Then,

${\hspace{2cm}}\gamma(P)\leq \gamma(P_L)\leq \tilde \gamma(P_L^*)=\tilde \gamma(I_{^mX,{\mathbb K}}\circ P_L^*)=\tilde \gamma(P^*),$

\noindent where the first equality follows from being $I_{^mX,{\mathbb K}}$ an isometric isomorphism.

\noindent Now Lemma \ref{epsilon} gives that
$$
\gamma(P^*)=\gamma(I_{^mX,{\mathbb K}}\circ P_L^*)=\gamma(P_L^*)\leq \tilde \gamma(P_L)=\beta(P_L(\overline B_{\otimes_{m,s}^{\pi_s}X}))=\beta(\overline \Gamma(P(B_X)))=\tilde \gamma(P).
$$

\noindent (2) By using part (1),

$\frac{1}{2}\gamma(P)\leq\frac{1}{2}\widetilde{\gamma}(P^*)=\frac{1}{2}\beta(P^*(\overline{B}_{Y^*}))\leq\frac{1}{2}\alpha(P^*(\overline{B}_{Y^*}))\leq\alpha(P^*(\overline{B}_{Y^*}))=\gamma(P^*)$

${\hspace{3cm}}\leq\widetilde{\gamma}(P)=\beta(P(\overline{B}_X))\leq\alpha(P(\overline{B}_X))\leq2\alpha(P(\overline{B}_X))=2\gamma(P)$,

\noindent and

$\frac{1}{2}\widetilde{\gamma}(P)=\frac{1}{2}\beta(P(\overline{B}_X))\leq\frac{1}{2}\alpha(P(\overline{B}_X))=\frac{1}{2}\gamma(P)\leq\frac{1}{2}\widetilde{\gamma}(P^*)\leq\widetilde{\gamma}(P^*)$

${\hspace{3cm}}=\beta(P^*(\overline{B}_{Y^*}))\leq\alpha(P^*(\overline{B}_{Y^*}))=\gamma(P^*)\leq\widetilde{\gamma}(P)\leq2\widetilde{\gamma}(P)$.
\end{proof}

As a consequence, we get R. Aron and M. Schottenloher result on compactness of polynomials:

\begin{corollary}{\cite[Proposition 3.6]{RMAMS}} Let $m\geq2$ and let $X$ and $Y$ be Banach spaces. Then for every homogeneous polynomial $P\in {\mathcal{P}}(^mX;Y)$ we have that
$P$ is compact if and only if its adjoint $P^*$ is compact.
\end{corollary}

The next result generalizes  \cite[Proposition 2]{DEEHOT} (see \cite[Theorem 3.1]{DLFMMEBS} for the multilinear case).

\begin{proposition} \label{este} Let $m\geq2$ and let $X$ and $Y$ be Banach spaces. Then

${\hspace{2cm}}a_n(P)\leq a_n(P^{**})+2\widetilde{\gamma}(P)$

\noindent for all $n\in\mathbb{N}$ and all $P\in\mathcal{P}(^mX;Y)$.
\end{proposition}
\begin{proof}
$\tilde a_n(P)=a_n(P_L)\leq a_n(P_L^{**})+2 \gamma(P_L)=a_n((I_{^mX,{\mathbb K}}\circ P^*)^*)+2\tilde \gamma(P)=a_n(P^{**}\circ I_{^mX,{\mathbb K}}^*)+2\tilde \gamma(P)=a_n(P^{**})+2\tilde \gamma(P).$
\end{proof}

\begin{corollary} Let $m\geq2$ and let $X$ and $Y$ be Banach spaces.
\begin{enumerate}
\item If $P\in {\mathcal{P}}(^mX;Y)$ is a compact operator, then $\tilde a_n(P)=a_n(P^*)$ for every $n\in\mathbb{N}$.
\item For every $P\in {\mathcal{P}}(^mX;Y)$ we have that $\tilde a_n(P)\leq 5a_n(P^*)$, $n\in\mathbb{N}$.
\end{enumerate}
\end{corollary}
\begin{proof} (1) Since $P^{*}$ is a linear continuous operator between Banach spaces, we have $a_n(P^{**})\leq a_n(P^*)$ (see, e.g., \cite[p. 152, 11.7.3. Proposition]{AP2}). If $P$ is compact, then $\widetilde{\gamma}(P)=0$ and hence, by Theorem \ref{este} and Proposition \ref{p:07.07.09}, we get

${\hspace{2cm}}a_n(P)\leq a_n(P^{**})+2\widetilde{\gamma}(P)=a_n(P^{**})\leq a_n(P^*)\leq \tilde{a}_n(P)$.

\noindent (2)  $\tilde a_n(P)=a_n(P_L)\leq 5a_n(P_L^*)=5a_n(I_{^mX,{\mathbb K}}\circ P^*)=5a_n(P^*)$.
\end{proof}

An alternative proof follows from the well known fact that $P$ is compact if and only if $P_L$ is compact (see \cite{RR2}) and the corresponding property for linear operators, that is, $\tilde a_n(P)=a_n(P_L)=a_n(P_L^*)=a_n(I\circ P^*)=a_n(P^*)$, for all $n\in \mathbb N$.

Let $P\in\mathcal{P}(^mX;Y)$.
The quantity $\displaystyle \tilde a(P):=\lim_{n\rightarrow\infty}\tilde a_n(P)\geq 0$ does not help when trying to measure the compactness of $P$. Even if $P$ is approximable (and so compact) whenever $\tilde a(P)=0$ the converse is, in general, not true.
 If we consider the approximation property (shortly, AP) on $Y$, then
any compact $m$-homogeneous polynomial $P\in\mathcal{P}(^mX;Y)$ can be approximated by finite-rank $m$-homogeneous polynomials (see \cite[Proposition 2.5]{GBLP}). Hence, similarly to the (multi)linear case, if the space $Y$ has the AP, then $P\in\mathcal{P}(^mX;Y)$ is compact if and only if $\tilde a(P)=0$. However,  by \cite[Proposition 3.3]{RMAMS} (see also \cite[Theorem 4.3]{JM2}) we know that $\mathcal{P}(^mX)$ has the AP if and only if, for every Banach space $Y$, the space of all finite-rank polynomials $\mathcal{P}_f(^mX;Y)$ is norm-dense in the space of all compact polynomials $\mathcal{P}_k(^mX;Y)$, or equivalently, any compact $m$-homogeneous polynomial $P\in\mathcal{P}(^mX;Y)$ can be approximated arbitrarily and closely by finite-rank $m$-homogeneous polynomials. Therefore if the space $\mathcal{P}(^mX)$ has the AP, then $P\in\mathcal{P}(^mX;Y)$ is compact if and only if $\tilde a(P)=0$. Let us remark that, there is a reflexive separable Banach space $X$ with basis such that $\mathcal{P}(^2X)$ does not have the AP (see \cite{RMAMS}). Hence, for this space $X$, which has the AP, there is a Banach space $Y$ such that there is a compact polynomial $P:X\longrightarrow Y$ which cannot be approximated by finite-rank polynomials. Note that it turns out that this space $Y$ also cannot have the AP by \cite[Proposition 2.5]{GBLP}.

As in the (multi-)linear case, we use  Kolmogorov numbers to measure how far a polynomial is from being compact.

We define the $n$-th Kolmogorov number $\tilde d_n(P)$ of a polynomial $P\in\mathcal{P}(^mX;Y)$ by

${\hspace{1cm}}\tilde d_n(P):=\inf\{\varepsilon>0: \ P(\overline{B}_X)\subset N_\varepsilon+\varepsilon\overline{B}_Y, \ N_\varepsilon\subset Y, \ dim(N_\varepsilon)<n\}$.

\noindent For $P:=T\in {\mathcal L}(X;Y)$ we write $d_n(T):=\widetilde d_n(P)$.

Kolmogorov numbers are related to approximation numbers via the equality $d_n(T)=a_n(TQ)$ , $n\in\mathbb{N}$, $T\in {\mathcal L}(X;Y)$. Recall that $Q$ is the canonical metric surjection from $l_1(\overline{B}_X)$ onto $X$, defined by

${\hspace{3.5cm}}\displaystyle Q(\{\lambda_x\})=\sum_{x\in\overline{B}_X}\lambda_xx$,  \ \ \ \ $\{\lambda_x\}\in l_1(\overline{B}_X)$

\noindent (see \cite[p. 150-151]{AP2}, and for the multilinear case see  \cite[Theorem 4.1]{DLFMMEBS}.) To get the polynomial version of this result we will use the next proposition  and the  study of the  lifting property for polynomials by Gonz\'alez and Guti\'errez \cite{MGJG}.

\begin{proposition}\label{dn}
Given $P\in {\mathcal P}(^mX;Y)$, $\tilde d_n(P)=d_n(P_L)$.
\end{proposition}

\begin{proof}
Clearly $\tilde d_n(P)\leq d_n(P_L)$. On the other hand, if $P(\overline B_X)\subset N_\epsilon +\epsilon \overline B_Y$ then $P_L(\overline B_{\otimes_{m,s}^{\pi_s}X})=\overline \Gamma (P(\overline B_X))\subset \overline \Gamma (N_\epsilon+\epsilon \overline B_Y)=\overline{N_\epsilon+\epsilon \overline B_Y}\subset N_\epsilon+(\epsilon+\delta)\overline B_Y$ for all $\delta>0$. Hence, $d_n(P_L)\leq \epsilon+\delta$ for all $\delta>0$ and so, $d_n(P_L)\leq \tilde d_n(P)$.
\end{proof}

\noindent As in the linear case, $P\in\mathcal{P}(^mX;Y)$ is compact if and only if $\displaystyle \tilde d(P):=\lim_{n\rightarrow\infty}\tilde d_n(P)=0$. Also it is obvious that $\tilde d_n(P)=0$ whenever $rank(P)<n$. Propositions \ref{I} and \ref{dn} imply that $\tilde d_n=d_n\circ I$ forms an $s$-number sequence.
 Proposition \ref{p:090.11.0} implies that $\tilde d_n(P)\leq \tilde a_n(P)$, for every $n\in \mathbb{N}$ and Proposition \ref{injsur} gives that
$(\tilde d_n)$ is a surjective $s$-number sequence.

Let $m\in\mathbb{N}$ and let $X$ be a Banach space. We say that $X$ has the {\it polynomial $m$-lifting property} if, for every continuous $m$-homogenous polynomial $P$ from $X$ to any quotient space $Y/N$, there is $\widetilde{P}\in\mathcal{P}(^mX;Y)$ such that $P=Q^Y_N\widetilde{P}$, where $Q^Y_N$ denotes the canonical map of $Y$ onto the quotient space $Y/N$.
We say that $X$ has the {\it polynomial metric $m$-lifting property} if, for every $\varepsilon>0$ and every continuous $m$-homogenous polynomial $P$ from $X$ to any quotient space $Y/N$, there is $\widetilde{P}\in\mathcal{P}(^mX;Y)$ such that $P=Q^Y_N\widetilde{P}$ and $\|\widetilde{P}\|\leq(1+\varepsilon)\|P\|$.

\begin{proposition}\label{lifting}
Let $m\in\mathbb{N}$. A Banach space $X$ has the polynomial (metric) $m$-lifting property if, and only if, $\hat\otimes_{m,s}^{\pi_s}X$ has the (respectively, metric) lifting property.
\end{proposition}

\begin{proof}
Assume first that $X$ has the polynomial $m$-lifting property. Let $T$ be a continuous linear operator from $\hat\otimes_{m,s}^{\pi_s}X$ into some quotient space $Y/N$. Let $P\in {\mathcal P}(^mX;Y/N)$ be such that $P_L=T$. By assumption, there is $\tilde P\in {\mathcal P}(^mX;Y)$ such that $Q^Y_N\circ \tilde P=P$. Since $Q^Y_N\circ (\tilde P)_L\circ \delta_X=P$, where $\delta_X$ is the $m$-homogeneous polynomial from $X$ to $\hat\otimes_{m,s}^{\pi_s}X$ given by $\delta_X(x)=x\otimes\cdots\otimes x$ (see \cite{RR2}), then $Q^Y_N\circ (\tilde P)_L=P_L=T$ and $\hat\otimes_{m,s}^{\pi_s}X$ has the lifting property.

We now assume  that $\hat\otimes_{m,s}^{\pi_s}X$ has the lifting property. Let $P\in {\mathcal P}(^mX;Y/N)$. Then $P_L\in {\mathcal L}(\hat\otimes_{m,s}^{\pi_s}X;Y/N)$. By assumption, there is $\widetilde{P_L}\in {\mathcal L}(\hat\otimes_{m,s}^{\pi_s}X;Y)$ such that $Q^Y_N\circ \widetilde{P_L}=P_L$. Then $\widetilde P:=\widetilde{P_L}\circ \delta_X$ satisfies $P=Q^Y_N\circ \widetilde P$.

The metric case follows from the fact that $\|P\|=\|P_L\|$.
\end{proof}

As a consequence, if $X$ has the polynomial metric $m$-lifting property, then for every $P\in\mathcal{P}(^mX;Y)$ we have
$\tilde d_n(P)=d_n(P_L)=a_n(P_L)=\tilde a_n(P),$ for all $n\in\mathbb{N}$.

\begin{theorem} \label{t:1.1.2} Let $m\geq2$ and let $X$ and $Y$ be Banach spaces. Let $P\in\mathcal{P}(^mX;Y)$, and let $Q$ be the canonical metric surjection from $l_1(\overline{B}_X)$ onto $X$. Then we have that

${\hspace{3.5cm}}\displaystyle \tilde d_n(P)=\tilde a_n(PQ)$ , $n\in\mathbb{N}$.
\end{theorem}
\begin{proof} By \cite[Theorem 1]{MGJG}  $l_1(\overline{B}_X)$ has the polynomial metric lifting property and so, from  Proposition \ref{lifting} $\hat\otimes_{m,s}^{\pi_s}l_1(\overline{B}_X)$ has the  metric lifting property. Then, $d_n(P_L\circ \otimes_m Q)=a_n(P_L\circ\otimes_mQ)$. Using that $(d_n)$ is surjective we get

$\tilde d_n(P)=d_n(P_L)=d_n(P_L\circ\otimes_mQ)=a_n(P_L\circ\otimes_mQ)=a_n((P\circ Q)_L)=\tilde a_n(P\circ Q).$
\end{proof}

Note that $(\tilde d_n)$ is the largest surjective $s$-number sequence. Indeed, given any surjective $s$-number sequence $(s_n)$, then for the canonical metric surjection $Q\in {\mathcal L}(l_1(\overline{B}_X);X)$ we have, by Theorem \ref{t:1.1.2}, that for any $P\in\mathcal{P}(^mX;Y)$

${\hspace{2.5cm}}\displaystyle s_n(P)=s_n(PQ)\leq \tilde a_n(PQ)=\tilde d_n(P)$.

By Propositions \ref{injsur} and \ref{dn}, $(\tilde d_n)$ is multiplicative. As a consequence for any surjective and multiplicative $s$-number sequence the following estimate holds for all $S\in {\mathcal L}(Y;Z)$ and all $P\in \mathcal{P}(^mX;Y)$:

${\hspace{3cm}}s_{k+n-1}(SP)\leq s_k(S)\tilde d_n(P)$, \ \ \ $k,n\in\mathbb{N}$,

\noindent In fact, Theorem \ref{t:1.1.2} yields

${\hspace{0.5cm}}s_{k+n-1}(SP)=s_{k+n-1}(SPQ)\leq s_k(S)s_n(PQ)\leq s_k(S)\tilde a_n(PQ)=s_k(S)\tilde d_n(P)$.\\

We end the paper with another example of s-number of homogeneous polynomials, namely, Gelfand numbers, from which we will get alternative characterizations of compactness of homogeneous polynomials.
Motivated by  \cite[ 11.5.1. Proposition]{AP2}  we define the Gelfand numbers  $\tilde c_n(P)$  of an $m$-homogeneous polynomial $P\in\mathcal{P}(^mX;Y)$  by

${\hspace{4cm}}\tilde c_n(P):=\tilde a_n(\kappa_YP).$

Clearly $(\tilde c_n)$ is an $s$-number sequence since $(\tilde a_n)$ is an $s$-number sequence, and for each $n\in\mathbb{N}$ we have that $\tilde c_n(P)\leq \tilde a_n(P)$. We will just write $c_n(T):=\tilde c_n(P)$ whenever $P=T\in {\mathcal L}(X;Y)$.
Note that $\tilde c_n(P)=c_n(P_L)$ for any  $P\in\mathcal{P}(^mX;Y)$. It follows from Proposition \ref{injsur} that $(\tilde c_n)$ is the largest injective $s$-number sequence, and satisfies the multiplicavity property (M). We also have a polynomial version of Carl's mixing multiplicavity of an injective $s$-number sequence $(s_n)$, that is, for all $S\in{\mathcal L}(Y;Z)$ and $P\in\mathcal{P}(^mX;Y)$, using Proposition \ref{p:090.11.0} we get that
$s_{k+n-1}(SP)\leq c_k(S)s_n(P)$, $k,n\in\mathbb{N}$.

Considering the function $\tilde c:\mathcal{P}(^mX;Y)\longrightarrow[0,\infty)$ given by
$\displaystyle \tilde c(P):=\lim_{n\longrightarrow\infty}\tilde c_n(P)$,
 we have that $\tilde c(P)=c(P_L)$. Now, compactness of homogeneous polynomials can be quantified by means  $\tilde c$ and $c$ as follows.

\begin{proposition} Let $m\geq2$ and let $X$ and $Y$ be Banach spaces. The following statements for a polynomial $P\in\mathcal{P}(^mX;Y)$ are equivalent.

(i) $P$ is compact.

(ii) $\tilde c(P)=0$.

(iii) $c(P^*)=0$.

\end{proposition}
\begin{proof} We know that $P$ is compact if and only if $P_L$ is compact (see \cite{RR2}), and $P$ is compact if and only if $P^*$ is compact (see \cite{RMAMS} or Corollary 4.4). Combining these facts with \cite[2.4.11]{AP4} we get the implications $(i)\Longleftrightarrow(ii)$, and $(i)\Longleftrightarrow(iii)$.
\end{proof}

Finally, following the lines of proof of \cite[Theorem 5.1]{DLFMMEBS} we get the following result, whose proof is omitted, which gives relation between Gelfand and Kolmogorov numbers of polynomials.

\begin{theorem} \label{t:1.1.3} Let $m\geq2$ and let $X$ and $Y$ be Banach spaces. Then, for every polynomial $P\in\mathcal{P}(^mX;Y)$ and  $n\in\mathbb{N}$ we have that

(i) $\displaystyle c_n(P^*)\leq \tilde d_n(P)$,

(ii) $\tilde c_n(P)=d_n(P^*)$,

(iii) $\tilde c_n(P)\leq2\sqrt{n} \ c_n(P^*)$.
\end{theorem}

\bigskip
{\bf Acknowledgement:}
The authors are deeply indebted to R. Aron, who proposed the research program on $s$-numbers for homogeneous polynomials and helped unselfishly to improve the paper.

\end{document}